\documentclass[12pt, twoside, leqno]{article}

\usepackage{amsmath,amsthm}
\usepackage{amssymb}

\usepackage[T1]{fontenc}

\numberwithin{equation}{section}

\begin{document}

\baselineskip=17pt

\newtheorem{thm}{Theorem}[section]
\newtheorem{cor}[thm]{Corollary}
\newtheorem{lem}[thm]{Lemma}
\newtheorem{lemma}[thm]{Lemma}
\newtheorem{prop}[thm]{Proposition}
\newtheorem{conj}[thm]{Conjecture}
\newtheorem{claim}[thm]{Claim}
\newtheorem{question}[thm]{Question}
\newtheorem{defi}[thm]{Definition}

\theoremstyle{definition}
\newtheorem{defn}[thm]{Definition}
\newtheorem{example}[thm]{Example}

\theoremstyle{remark}
\newtheorem{rmk}[thm]{Remark}

\newcommand{\BA}{{\mathbb{A}}}
\newcommand{\BB}{{\mathbb{B}}}
\newcommand{\BC}{{\mathbb{C}}}
\newcommand{\BD}{{\mathbb{D}}}
\newcommand{\BE}{{\mathbb{E}}}
\newcommand{\BF}{{\mathbb{F}}}
\newcommand{\BG}{{\mathbb{G}}}
\newcommand{\BH}{{\mathbb{H}}}
\newcommand{\BI}{{\mathbb{I}}}
\newcommand{\BJ}{{\mathbb{J}}}
\newcommand{\BK}{{\mathbb{K}}}
\newcommand{\BL}{{\mathbb{L}}}
\newcommand{\BM}{{\mathbb{M}}}
\newcommand{\BN}{{\mathbb{N}}}
\newcommand{\BO}{{\mathbb{O}}}
\newcommand{\BP}{{\mathbb{P}}}
\newcommand{\BQ}{{\mathbb{Q}}}
\newcommand{\BR}{{\mathbb{R}}}
\newcommand{\BS}{{\mathbb{S}}}
\newcommand{\BT}{{\mathbb{T}}}
\newcommand{\BU}{{\mathbb{U}}}
\newcommand{\BV}{{\mathbb{V}}}
\newcommand{\BW}{{\mathbb{W}}}
\newcommand{\BX}{{\mathbb{X}}}
\newcommand{\BY}{{\mathbb{Y}}}
\newcommand{\BZ}{{\mathbb{Z}}}

\newcommand{\CA}{{\mathcal A}}
\newcommand{\CB}{{\mathcal B}}
\newcommand{\CC}{{\mathcal C}}
\newcommand{\CD}{{\mathcal D}}
\newcommand{\CE}{{\mathcal E}}
\newcommand{\CF}{{\mathcal F}}
\newcommand{\CG}{{\mathcal G}}
\newcommand{\CH}{{\mathcal H}}
\newcommand{\CI}{{\mathcal I}}
\newcommand{\CJ}{{\mathcal J}}
\newcommand{\CK}{{\mathcal K}}
\newcommand{\CL}{{\mathcal L}}
\newcommand{\CM}{{\mathcal M}}
\newcommand{\CN}{{\mathcal N}}
\newcommand{\CO}{{\mathcal O}}
\newcommand{\CP}{{\mathcal P}}
\newcommand{\CQ}{{\mathcal Q}}
\newcommand{\CR}{{\mathcal R}}
\newcommand{\CS}{{\mathcal S}}
\newcommand{\CT}{{\mathcal T}}
\newcommand{\CU}{{\mathcal U}}
\newcommand{\CV}{{\mathcal V}}
\newcommand{\CW}{{\mathcal W}}
\newcommand{\CX}{{\mathcal X}}
\newcommand{\CY}{{\mathcal Y}}
\newcommand{\CZ}{{\mathcal Z}}

\newcommand{\ra}{{\ \longrightarrow\ }}
\newcommand{\la}{{\ \longleftarrow\ }}
\newcommand\Hom{\operatorname{Hom}}
\newcommand\Id{\operatorname{Id}}
\newcommand\id{\operatorname{id}}
\newcommand\Spec{\operatorname{Spec}}
\newcommand\End{\operatorname{End}}
\newcommand\Aut{\operatorname{Aut}}

\newcommand{\an}{{\mathrm{an}}}
\newcommand{\val}{{\mathrm{val}}}
\newcommand{\Pic}{\mathop{\rm Pic}\nolimits}

%%%%%%%%%%%%%%%%

\title{Equidistribution of torsion points in abelian varieties}

\author{Jiyao Tang\\Peking University\\ 
Beijing 100871, China\\
E-mail: JiyaoTang@pku.edu.cn}

\date{}

\maketitle

%% Classification and key words; note that the 2010 classification is used:

\renewcommand{\thefootnote}{}

%\footnote{2020 \emph{Mathematics Subject Classification}: Primary XXXX; Secondary YYYY.}

%\footnote{\emph{Key words and phrases}: aaaa, bbbb, cccc.}

\renewcommand{\thefootnote}{\arabic{footnote}}
\setcounter{footnote}{0}

%%%%%%%%

\begin{abstract}
Let $A$ be an abelian variety over a field $K$ with a complete non-Archimedean norm. We prove that the torsion points of $A$ are equidistributed over $A^\an$ with respect to the canonical measure.
\end{abstract}

\section{Introduction}
\subsection{Motivation}
When $A$ is an abelian varieties over the field of complex numbers $\BC$, it's isomorphic to $\BC^g/\Gamma$ for a positive integer $g$ and a complete lattice $\Gamma\subset \BC^g$.  

Let $f$ be a real-valued continuous funtion on $A\cong \BC^g/\Gamma$, here $f$ is continuous in the complex topology. Then by the definition of Riemann integral, we know
$$
\lim_{m\rightarrow \infty}\frac{1}{m^{2g}}{\Sigma_{x\in A[m]}f(x)}=\int_{A}f\mathrm{d}\mu
$$
where $A[m]$ is the set of the $m$-torsion points of $A$ and $\mathrm{d}\mu$ is the canonical Haar measure.

When the base field changes to number field, Szpiro-Ullmo-Zhang \cite{SUZ} proved the equidistribution problem under the setting of Arakelov geometry. The equidistribution theorem of \cite{SUZ} is crucial in the proof of the Bogomolov conjecture in number fields by Ullmo \cite{Ul} and Zhang \cite{Zh}.

When the base field is a non-archimedean field, Chambert-Loir \cite{Ch} defined the canonical measure over Berkovich spaces and proved an equidistribution theorem \cite[Theorem 3.1]{Ch}. It is a non-archimedean analogue of Szpiro-Ullomo-Zhang's theory. 

And in \cite{Yuan}, Yuan proved a generalised version of the equidistribution theorem in \cite{SUZ} by a bigness theorem in the setting of Arakelov theory as an arithmetic analogue of a classical theorem of Siu.

Gubler \cite{Gu1} and \cite{Gu2} also gives an explicit description of the canonical measure in terms of tropical geometry. And it is a key step in the proof of Bogomolov conjecture for totally degenerate abelian varieties over function fields \cite{Gu3}.

GIGNAC \cite{Will} proved the case of good reduction.

My work implies the result for abelian varieties over local fields and can be descended to number fields.

\subsection{Terminology and statement of the problem}
Let $K$ be a non-arichimedean field, $R$ its valuation ring and $k$ the residue field.

Let $A$ be an abelian variety and $\CA$ the N\'eron model of $A$ over $R$. Let $\CA_k$ be the corresponding special fibre.

For any scheme $X$ over $K$, later $X^\an$ will be the corresponding Berkovich space. 

By \cite{SW}, we have the Raynaud extension
\begin{equation}\label{Raynaud1} 
0\ra T \ra E \ra B \ra 0   \tag{$\ast$}
\end{equation}
where $T$ is an algebraic torus over $K$, $B$ is an abelian variety of good reduction over $K$ and $E/M=A^\an$ for a lattice $M$ of rank $r$=rank $T$. $\pi:E\ra A^\an$ will refer to the projection map (a morphism between Berkovich spaces).

Also by \cite{SW},  we know the Raynaud extension (\ref{Raynaud1}) locally on $B^\an$ trivial, i.e, there is an affinoid open cover $\{U_\alpha\}$ of $B^\an$ such that $E|_{U_\alpha}=T^\an \times U_\alpha$ thus we can define the valuation map, $\val:E|_{U_\alpha}\cong T^\an \times U_\alpha\ra\BR^r$, where $r$ is the rank of $T$, also the embedding morphism of the skeleton $\tau:\BR^r \ra E$ and we know $\val(M)$ is  a lattice in $\BR^r$ of rank $r$.

By $g$, $r$, $s$ we mean the dimention of $X$, $T$, $B$ respectively (thus $g=r+s$). 
We want to prove the following theorem:
\begin{thm}\label{main}
	Let $f$ be a continuous funtion on $A^\an$, we have 
	$$
	\lim_{m\rightarrow \infty}\frac{1}{m^{2g}}\Sigma_{x\in A[m]}\deg (x)f(x)=\int_{A^\an}f\mathrm{d}\mu
	$$
	where $m$ is invertible in $K$ and $\mathrm{d}\mu$ is the canonical measure in $A^\an$.
\end{thm}
\begin{rmk}
	The condition $m$ is invertible $K$ is necessary, for a counterexample we can consider a supersingular elliptic curve $E$ with good reduction, in this case, the canonical measure is dirac measure of the divisorial point $x_G$, but $0\in E$ is the only torsion point and $x\ne x_G$.
\end{rmk}
For an $m$-torsion point $x$ of $A$, it admits $\deg(x)$ preimages in $A_{\widehat{\bar{K}}}$, thus the proposition for $A_{\widehat{\bar{K}}}$ implies that for $A$, so without of generality, we may assume the field $K$ is alegbraically closed, then all algebraic torus will be splitting torus. And latter we will always have the assumption that $m$ is invertible in $K$.

\subsection{Valuation map and skeleton}
For a splitting torus $T$, we have the valuation map:
$$
\val:T^\an\ra \Hom(\hat{T},\BR), x\mapsto (\chi\mapsto -\log|\chi|_x)
$$
Where $\hat{T}$ is the character group.
Now fix a basis ${T_i}_{i=1}^r$ of $\hat{T}$, we have an isomorphism $\Hom(\hat{T},\BR)\cong \BR^n$. Then $T=\Spec[T_1,\cdots,T_r,T_1^{-1},\cdots,T_r^{-1}]$ and the valuation map can be expressed as:
$$
\val:T^\an\ra \BR^r, x\mapsto (-\log|T_1|_x,\cdots,-\log|T_r|_x)
$$
For an abelian variety $A$, we have the Raynaud extension,
\begin{equation} 
0\ra T \ra E \ra B \ra 0   
\end{equation}
And by \cite{SW} (the following will also be proved in section \ref{GeneralCase}), there is an affnoid covering $\{U\}$ of $B$, such that $E|_U^\an=T^\an \times U^\an$, then we can define the valuation map on $E|_U^\an$ by composition of:
$$
E|_U^\an\ra  T^\an \stackrel{\val}{\ra} \BR^r
$$ 

Because the transition map $E|_{U\cap V}\cong T\times (U\cap V) \ra T\times (U\cap V)\cong E|_{U\cap V}$ does not change the value (see \cite{SW}) of $|T_i|$ so we have a well-defined map $\val:E^\an \ra \BR^r$, and pass to a map $\val:A^\an\ra \BR^r/\Gamma$, where $\Gamma$ is a complete lattice in $\BR^r$. By \cite{Berk}, this map is a deformation retract, and we denote the embedding by $\sigma: \BR^r\ra E^\an$.

Then the Haar measure in $\BR^r/\Gamma$ pushforward to a measure in $E^\an/M\cong A^\an$, this is the canonical measure in $A^\an$     

\section{Good reduction case}\label{GoodReductionCase}
In this section, we suppose the abelian variety $A$ has good redution. Let $x_G$ be the Gauss point.

At first, we note that $A$ can be covered by finite affine opens $U_i$, and we can cover each $U_i^\an$ with some $V_{i_j}^\circ$, where $V_{i_j}$ are affinoid domains. By the compactness of $A^\an$, we can find a finite compact cover ${V_j}$ ($j=1,\cdots,n$) of $A^\an$, and each $V_j$ is contained in an affine open subset $U_j$ of $A^\an$. Substistude $V_j$ by $V_j\cup \{x_G\}$ we may assume $x_G\in V_j$ (we will only use the fact that $V_j$ is compact and contained in an affine open).

Then we show funtions in $V_j$ can be approximated by algebraic functions:
\begin{defi}
	For $V_j\subset U_j^\an$ where $U_j$ is an affine scheme $\Spec B$, for $f\in B$, we call $V_j\ra \BR: x\mapsto |f|_x$ the corresponding algebraic function.
\end{defi}  
We need the following Stone-Weierstrass theorem \cite[$P121-P122$]{WR}:
\begin{thm}[Stone-Weierstrass]\label{Stone-Weierstrass}
	If $X$ is a compact and Hausdorff space, denote by $C(X)$ the Banach algebra of all continuous $\BR$-functions on $X$. Let $A$ be a subalgebra such that:\\
	(a) $A$ separates points on $X$, i.e, for different points $x,y\in X$, there is an $f\in A$ such that $f(x)\ne f(y)$.\\
	(b) At every $p\in X$, $f(p) \ne 0$ for some $f\in A$.\\
	Then $A$ is dense in $C(X)$.
\end{thm}
Now denote $L_j$ to be the $\BR$-algebra generated by algebraic funcions over $V_j$. By theorem \ref{Stone-Weierstrass}, we know the $L_j$ is dense in $C(V_j)$. 

\begin{claim}\label{GoodReductionClaim}
	Let $U$ be an affine open of $A$, and $f$ be an algebraic function in $U^\an\subset A^\an$, then $f(x_G)\ne f(x)$ at only $O(m^{2g-2})$ points, where $x_G$ is the Gauss point of $X^\an$ and $g$ is the dimension of $A$.
\end{claim}

We will need the following lemma:
\begin{lemma}\label{GoodRedutionCaseLem}
	Let $A$ be an abelian variety of dimension $g$ over a field $k$, $Z$ is a non-trivial closed subset of $A$, then there are only $O(m^{2g-2})$ (related to $Z$) $m$-torsion points in $Z$ .
\end{lemma}
\begin{proof}
	Without loss of generality, we may assume $Z$ is irreducible and of degree $k<g$, let $[m]$ be the mutiplication by $m$ in $A$. Then we choose a very ample and symmetric bundle $\CL$ on $A$. By \cite[Section 6]{AV}, we have $m^*\CL=\CL^{\otimes m^2}$. Choose $H$ to be the interstion of $k$ general divisors containing the indentity $O\in X$ corresponding to $\CL$, then by projection formula
	\begin{equation}
	\begin{split}
	\#(Z\cap A[m])
	& \le([m]_*Z \cdot H)\\ 
	& =(Z\cdot [m]^*H) \\
	& =(Z\cdot [m]^*\CL \cdots [m]^*\CL)\\
	& =m^{2k}(Z\cdot \CL\cdots\CL)
	\end{split}
	\end{equation}
	Then number of the $m$-torsion points in $Z$ is less than $m^{2k}(Z\cdot \CL\cdots\CL)$, which is $O(m^{2g-2})$.
\end{proof}

Now let $f$ be an algebraic function on $U^\an=(\Spec A)^\an$ corresponding to $h\in A$. We may assume $\mathrm{ord}_{\CA_k}h=0$. Then $f(x_G)=1$ and for any point $x\in A^\an$, $f(x)\ne 1$ if and only if $\sigma(x)\in V(\tilde{h})$ where $\sigma: A^\an \ra \CA_k^\an$ is the redution map and $\tilde{h}$ is the redution of $h$.

Because $m$ is not divisible by char $K$, $m$ is thus invertible in the formal group so the redution map induces injection of the $m$-torsions of $A$ into the $m$-torsion of $\CA_k$, and by lemma \ref{GoodRedutionCaseLem} we know there are only $o(m^{2g})$ $m$-torsion points on $V(\tilde{h})$, thus $o(m^{2g})$ $m$-torsion points $x$ satisfies $f(x)\ne f(x_G)$. 

By previous discuss, we know $L_j$ is dense in $C(V_j)$. For any function $f\in C(A^\an)$, suppose $\|f-f_{j}\|<\epsilon$ where $f_j\in L_j$ are linear combination of algebraic functions on $V_j$. Let $N=\max_j(\|f\|,\|f_{j}\|)$, and for only $o(m^{2g})$ points, we have $f_{j_i}(x)\ne f_{j_i}(x_G)$. So
\begin{equation}
\begin{split}
&\left|\frac{1}{m^{2g}}\Sigma_{x\in A[m]}f(x)-f(x_G)\right|\\
=&\left|\frac{1}{m^{2g}}\Sigma_{x\in A[m]}(f(x)-f(x_G))\right|\\
\le & \Sigma_{j=1}^{n}\frac{1}{m^{2g}}\Sigma_{x\in A[m]\cap V_j}(|f(x)-f_{j}(x)|+|f_{j}(x)-f_{j}(x_G)|+|f_{j}(x_G)-f(x_G)|)\\
\le& 2n\epsilon+\Sigma_{j=1}^{n}\frac{1}{m^{2g}}o(m^{2g})N\\
=& 2n\epsilon+o(1)
\end{split}   	
\end{equation}   
Because $\epsilon$ is arbitrary, this term converges to zero when $m\rightarrow \infty$, which is we want to show.

\section{Totally degenerate case}\label{TotallyDegenerateCase}
In this section, we suppose the abelian vareity $A$ is totally degenerate, i.e, $A^\an$ can be expressed as a quotient $T^\an/M$ where $T$ is a algebraic torus of rank $r$, $M$ is a lattice of rank $n$. By the assumption that $K$ is algebraically closed, $T$ is a spliting torus. And we denote by $\pi: T^\an \ra A^\an$ the projection map.

Then $T^\an$ can be expressed as the set of semi-norms on $K[t_1,\cdots,t_n,t_1^{-1},\cdots,t_n^{-1}]$, we can choose a rational domain $T_0=\{x\in T^\an \big| q^{-1}\le|t_i|_x\le q\}$($q>1$), such that $\pi(T^\an)=A^\an$. Because $T_0$ is an affinoid domain, it is compact. 

We then prove the following lemma:
\begin{lem}\label{lem:tiny}
	For $f=\Sigma_{v\in \BZ^r}a_v t^v \in K[t_1,\cdots,t_r,t_1^{-1},\cdots,t_r^{-1}]$ where $a_v=0$ for all but finitely many $v\in \BZ^r$. We have $|f|_x\ne |f|_{\tau(val(x))}$  only for $o(m^{2r})$ pre-$m$-torsion points $x\in T(K)$, where by pre-$m$-torsion points we mean those being sent to $m$-torsion points by $\pi$.
\end{lem}
\begin{proof}
	If $f=0$ we are done. Now assume $f\ne 0$. 
	
	Because the norm is non-archimedean, we know for $|t_i|_x=l_i $, if $|a_v| l^v\ne |a_{v'}| l^{v'}$ for all $a_v,a_{v^\prime}\ne 0$, we have  $|f|_x=\max\limits_{v\in \BZ^r}{|a_v|l^v}$ which only depends on the value of $l_i=|t_i|_x$. 
	
	Those $(l_i)$'s satisfying $|a_v| l^v= |a_{v'}| l^{v'}$ for some $a_v,a_{v^\prime}\ne 0$ form finitely many hyperplanes in $\val(T_0)\subset \BR^r$, which can only contain the image of $o(m^{2r})$  pre-$m$-torsion points.
	
	Note that $|t_i|_x=|t_i|_{\tau(val(x))}$ (the valuation map is just $x\mapsto (|t_1|_x,\cdots,|t_r|_x)$), thus we are done.   
\end{proof}

Now let $C$ be the algebra generated by the functions mentioned in lemma \ref{lem:tiny}, then $C$ is dense in $C(T_0)$ by theorem \ref{Stone-Weierstrass}. Becasuse functions in $C$ are finite linear combinations of $|f|$ for $f=\Sigma_{v\in \BZ^n}a_v t^v \in K[t_1,\cdots,t_n,t_1^{-1},\cdots,t_n^{-1}]$, we know for any $g\in C$, we have $|g|_x\ne |g|_{\tau(val(x))}$  only for $o(m^{2r})$ pre-$m$-torsion points. By the compactness of $T_0$, $g$ is bounded, we thus have 
\begin{equation}\label{eq:1}
\frac{1}{m^{2r}}\Sigma_{x\in T_0[m]}(g(x)-g(\tau(val(x)))=o(1) 
\end{equation}
here $T_0[m]$ consists of one representative for each $m$-torsion points in $T_0$, because $C$ is dense in $C(T_0)$, the equation holds for all $g\in C(T_0)$.

Now let $h$ be a continuous function on $A^\an$, then $h'=\pi|_{T_0}\circ h:T_0\ra \BR$ is a continuous function, thus
\begin{equation}
\begin{split}
&\left|\frac{1}{m^{2r}}\Sigma_{x\in A[m]}h(x)-\int_{A^\an} h(x)d\mu\right|\\
=&\left|\frac{1}{m^{2r}}\Sigma_{x\in T_0[m]}h'(x)-\int h(x)d\mu\right|\\
\le&  \left|\frac{1}{m^{2r}}\Sigma (h'(x)-h'(\tau(val(x))))\right|\\
+&\left|\frac{1}{m^{2r}}\Sigma h'(\tau(val(x)))-\int h(x)d\mu\right|\\
=&o(1)+ \left|\frac{1}{m^{2r}}\Sigma h'(\tau(val(x)))-\int h(x)d\mu\right|\\
=&o(1)+ \left|\frac{1}{m^{2r}}\Sigma_{x\in A[m]} h(\tau(val(x)))-\int h(x)d\mu\right|\\
=&o(1).
\end{split}
\end{equation}

The last equation is because $\val(T_0[m])$ are equidistributed in the skeleton of $A^\an$.

So we know the $m$-torsion points are equidistributed in $A$.

\section{General case}\label{GeneralCase}
Now by \cite{SW}, we have the Raynaud extension
\begin{equation}
0\ra T^\an \ra E \ra B^\an \ra 0  
\end{equation} 
Let $r,g,s$ be the dimension of $T,E,B$ respectively. Because $K$ is algebraically closed, $T$ is a splitting torus. To go on, we wish $E$ is `almost' the product of $T^\an$ and $B^\an$. The following lemma is needed:
\begin{lem}
	Let $0\ra (\BG_m)^n \ra E \ra A \ra 0$ be an extension of an abelian variety by a split torus, then there is an affine open cover ${U_i=\Spec C_i}$ of $A$ over which $G$ is trivial, i.e, $E|_{U_i}=(G_m)^n\times U_i$.
\end{lem}
\begin{proof}
	By the exactness, $G$ is a principal homogeneous space (or a torsor) of $\BG_m^n$ over $A$. Because for any scheme $X$, $H^1(X_{et},\BG_m^n)\cong H^1(X_{zar},\CO_X^*)^n= \Pic(X)^n$, so it suffices to find an open cover such that the image of $E$ is trivial in $\Pic(X)^n$, then an affine refinement for it is what we need.
	
	Suppose $E$ is send to $(\CL_1,\cdots,\CL_n)$, then for any point $x$ of $A$, we can find a prime divisor $Z$ of $A$ not passing $x$, lies in the linear equivalent class of $\CL_i$, set $U=A\setminus Z$, then we have the following exact sequence:
	$\BZ\ra \Pic(A) \ra \Pic(U) \ra 1$, where the first map is $1\mapsto Z$, the second is $D\mapsto D\cap U$. Then $\CL_i$ is trivial in $\Pic (U)$. Do this for every $i=1,\cdots,n$, we thus have $E$ is trivial in $\Pic(U)^n$.	  
\end{proof}

By similar discussion as in section \ref{GoodReductionCase} we can choose a finite compact cover $V_j$ ($j=1,\cdots,n$) of $B^\an$, each $V_j$ is a compact subset lies in a (Zariski) affine open $U_j^\an$ of $B^\an$ and the Gauss point $x_G$ of $B^\an$ is in $V_j$, such that the Raynaud extension is trivial over $U_j$. We choose $T_0$ as Laurent domain in $T^\an$ with finite radius such that $\pi(\cup_{j=1}^{n}T_0\times V_j)$ can cover the whole $A^\an$. 

We then construct a family of functions $f$ in $T_0\times V_j$ satisfies $f(x)\ne f(\tau(val(x)))$ only for $o(m^{2g})$ pre-$m$-torsion points in $T_0\times V_j$.

Now let $L_j\subset C(T_0\times V_j)$ be the $\BR$-subalgebra generated by algebraic functions on the analytication of $T\times U_j=\Spec C_j[t_1,\cdots,t_r,t_1^{-1},\cdots , t_r^{-1}]$ (they can restrict to $T_0\times V_j$). 

For an algebraic function $f$ corresponding to $0\ne h=\Sigma_{v\in \BZ^r}a_v t^v\in C_j[t_1,\cdots,t_r,t_1^{-1},\cdots,t_r^{-1}]$, with $a_v\in C_j$, the locus of $U_j$ where all $a_v$ vanish form a closed subscheme of $U_j$, so it can only contain $o(m^{2s})$ $m$-torsion points of $B$ by the discussion of section \ref{GoodReductionCase}. Over the fibers of them, there are only $o(m^{2g})$ pre-$m$-torsion points of $A^\an$ in $T_0\times U_j$. And for an $m$-torsion point of $B$, $x\in U_j^\an\subset B^\an$, by the discussion of \ref{TotallyDegenerateCase} there are only $o(m^{2r})$ pre-$m$-torsion points $y$ of $A$ over $x$ belongs to $T_0\times {x}$, such that $f(y)\ne f(\tau(\val(y)))$. So there are only $o(m^{2g})$ pre-$m$-torsion points $x\in T_0\times V_j$ of $A^\an$ such that $f(x)\ne f(\tau(\val(x)))$.  

Because functions in $L_j$ are just finite linear combinations of algebraic functions, they also satisfies the above condition. And by Stone-Weierstrass theorem \ref{Stone-Weierstrass}, $L_j$ is dense in $C(T_0\times V_j)$.

Then for any continuous function $f$ on $A^\an$, suppose $N=max_{x\in A^\an}f(x)$ and any $\epsilon>0$ we can choose $g_j\in L_j$ such that $\|\pi^*f-g_j\|<\epsilon$ (so $\|g_j\|\le N+\epsilon$), and choose $N_0\in \BN$ such that for $m>N_0$, we have $g_j(x)\ne g_j(\tau(val(x)))$ only for pre-$m$-torsion points of number less than $\epsilon m^{2g}$ and $N_1$ such that for $m>N_1$, $|\int_{A^\an} fd\mu-\frac{\Sigma_{x\in A[m]}f(\tau(\val(x)))}{m^{2g}}|<\epsilon$. Then for any $m>N_0,N_1$
\begin{equation}
\begin{split}
&\left|\frac{1}{m^{2g}}\Sigma_{x\in A[m]} f(x)-\int_{A^\an} fd\mu\right|\\
\le&\epsilon+\left|\Sigma_{x\in A[m]}\frac{1}{m^{2g}}f(x)-f(\tau (\val(x)))\right|\\
=  &\epsilon+\left|\Sigma_{x\in E[m]}\frac{1}{m^{2g}}\pi^*f(x)-\pi^*f(\tau (\val(x)))\right|\\
\le&\epsilon+\Sigma_{j=1}^n\left|\Sigma_{x\in E_j[m]}\frac{1}{m^{2g}}\pi^*f(x)-\pi^*f(\tau (\val(x)))\right|\\
=  &\epsilon+\Sigma_{j=1}^n\left|\frac{1}{m^{2g}}\Sigma (\pi^*f(x)-g_j(x)+g_j(x)-g_j(\tau(\val(x)))+g_j(\tau(\val(x)))-\pi^*f(\tau (\val(x)))\right|\\
\le&\epsilon(2n+1)+\Sigma_{j=1}^n\left|\Sigma\frac{1}{m^{2g}} (g_j(x)-g_j(\tau(\val(x))))\right|\\
\le&\epsilon(2n+1)+(N+\epsilon)n\epsilon.
\end{split}
\nonumber
\end{equation}
where $E[m]$, $E_j[m]$ are classes of representatives of pre-$m$-torsion points in $E$ and $T_0\times V_j$ respectively. Then by the generality of $\epsilon$, we know when $m\ra \infty$, this term converges to $0$, thus torsion points are equidistributed in $A^\an$.

\section{Acknowledgements}
      This paper is the author's Undergraduate Research taken in Peking University.
      
      \noindent
      The author is deeply indebted to Professor Xinyi Yuan, who introduced him to this subject and encouraged him throughout this project. He is also grateful to Professor Junyi Xie, with whom the dicussions are very useful to the author.


\clearpage
\begin{thebibliography}{20}
	
	\bibitem[Berk]{Berk}
	Vladimir G.Berkovich
	{\em Spectral theory and Analytic Geometry Over Non-Archimedean Fields}
	American Mathematical Society, 1990
	
	\bibitem[Ch]{Ch}
	Chambert-Loir, A.
	{\em Mesures et équidistribution sur les espaces de
		Berkovich}, J.Reine Angew. Math. 595, 215-235 (2006)
	
	\bibitem[Gig]{Gig}
	William Gignac
	{\em Equidistribution of preimages over nonarchimedean fields for maps of good reduction}, Ann. Inst. Fourier, Gerenoble. 64, 4, 1737-1779 (2014)
	
	
	\bibitem[Gu1]{Gu1}
	Gubler, W.
	{\em Non-archimedean canonical measures
		on abelian varieties}, Compositio Math. 146 (2010), 683–730
	
	\bibitem[Gu2]{Gu2}
	Gubler, W.
	{\em Tropical varieties for non-archimedean analytic
		spaces}, Invent. math. 169, 321–376 (2007)
	
	\bibitem[Gu3]{Gu3}
	Gubler, W.
	{\em The Bogomolov conjecture for totally degenerate abelian varieties}, Invent. Math. (2007).
	
	\bibitem[Mumf]{AV}
	David Mumford, C.P.Ramanujam, Yuri Manin
	{\em Abelian Varieties},
	Tata Institute of Fundamental Research, Mumbai
	
	\bibitem[Rudi]{WR}
	Walter Rudin
	{\em Functional analysis},
	International Series in Pure and Applied Mathematics, (McGraw-Hill, Inc. Second edition).
	
	\bibitem[SUZ]{SUZ}
	Szpiro, L., Ullomo, E., Zhang, S. {\em \'Equirepartition des petits points.} Invent. Math. 127,337-348(1997)
	
	\bibitem[SW]{SW}
	Siegfried Bosch, Werner L\"utkebohmert,
	{\em Degenerating Abelian Varieties}, Topology Vol.30, No 30, 653-698, 1991
	
	\bibitem[Ul]{Ul}
	Ullomo, E.,
	{\em Positivit\'e et discr\'etion des points alg\'ebraques des courbes}.
	Ann. Math. (2) 147, 167-179 (1998)
	
	\bibitem[Vaki]{RV}	
	Ravi Vakil,
	{\em The Rising Sea, Foundations of Algebraic Geometry},
	math216.wordpress.com
    
    \bibitem[Will]{Will}
    William GIGNAC,
    {\em Equidistribution of preimages over nonarchimedean fields for maps
    	of good reduction}, Ann. Inst. Fourier, Grenoble.
      64, 4 (2014) 1737-1779
    
     
	\bibitem[Yuan]{Yuan}
	Xinyi Yuan, 
	{\em Big line bundles and arithmetic varieties.} Invent math. 173, 603-649(2008)
	
	\bibitem[Zhan]{Zh}
    Shouwu Zhang,
	{\em Equidistribution of small points on abelian varieties.} Ann. Math. 147(1), 159-165 (1998)
	
\end{thebibliography}
\end{document}